\documentclass[12pt,reqno]{amsart}

\usepackage{amsmath,amssymb,amsthm,amsfonts}
\usepackage{mathtools}
\usepackage{geometry}
\usepackage{hyperref}
\usepackage{enumitem}
\usepackage{url}

\newtheorem{theorem}{Theorem}[section]
\newtheorem{lemma}[theorem]{Lemma}
\newtheorem{proposition}[theorem]{Proposition}

\newtheorem{definition}[theorem]{Definition}
\newtheorem{example}[theorem]{Example}
\newtheorem{remark}[theorem]{Remark}

\newcommand{\R}{\mathbb{R}}

\title{Anisotropic Parabolic Obstacle Problems and the Stefan Problem:
Regularity of the Evolving Free Boundary}

\author{Ezequiel Barbosa* \and Rosivaldo Gonçalves** \and Luan de Figueiredo*}
\thanks{*Departamento de Matem\'atica, Universidade Federal de Minas Gerais (UFMG), Av. Antônio Carlos 6627, Caixa Postal 702, 30123-970, Belo Horizonte, MG, Brazil. 
Emails: \texttt{ezequiel@mat.ufmg.br}, \texttt{luan@mat.ufmg.br}}
\thanks{**Centro de Ciências Exatas, Universidade Estadual de Montes Claros (Unimontes), Emails: \texttt{rosivaldo.goncalves@unimontes.br}}
\date{}
\begin{document}
\begin{abstract}
We study a parabolic obstacle problem for surfaces evolving by anisotropic
mean curvature flow subject to an obstacle constraint. Given a convex obstacle and initial
data, we seek an evolving surface minimizing an anisotropic energy functional while remaining above the obstacle; as a special case, this framework includes the anisotropic Stefan
problem, where the free boundary represents a phase transition interface with direction-dependent surface tension. The central tool is the Cahn--Hoffman transform $S(x) = A^{-1/2}x$,
which maps the Wulff ellipsoid $\{x : x^T A^{-1}x \leq 1\}$ to the Euclidean unit ball and converts
the anisotropic problem into an equivalent isotropic one with a generalized Robin-type condition on the free boundary. We prove optimal regularity of the solution ($C^{1,\alpha}$ in space and
$C^{0,\alpha/2}$ in time up to the free boundary) and $C^{1,\alpha}$-regularity of the evolving free boundary
at non-degenerate points. The parabolic Hausdorff dimension of the space-time singular set
is shown to be at most $n - 1$.
\end{abstract}

\maketitle

\section{Introduction and Main Results}

\subsection{Historical context and motivation}

The mathematical study of moving interfaces in phase transitions dates back to the Stefan
problem, formulated by Josef Stefan in 1891 to describe the melting of ice. In its simplest
form, one seeks a temperature distribution $u(x,t)$ and a moving boundary $\Gamma(t)$ separating
the solid and liquid phases, satisfying
\begin{equation}\label{eq:stefan-classical}
\begin{cases}
u_t - \Delta u = 0 & \text{in } \Omega(t), \\
u = 0 & \text{on } \Gamma(t), \\
V_\Gamma = -\partial_\nu u & \text{on } \Gamma(t),
\end{cases}
\end{equation}
where $V_\Gamma$ is the normal velocity of the free boundary and $\Omega(t)$ is the liquid region. The
condition on $V_\Gamma$ expresses the physical principle that latent heat released at the interface
drives its motion. For comprehensive treatments, see Friedman \cite{Friedman1982} and the surveys of
Caffarelli \cite{Caffarelli1977, Caffarelli1978}.

The regularity theory for the Stefan problem was developed by Caffarelli \cite{Caffarelli1977, Caffarelli1978}, who
observed that it can be reformulated as an obstacle problem in space-time: the temperature
$u$ satisfies a variational inequality, and the free boundary is the boundary of the coincidence
set. This reformulation allowed the application of elliptic regularity machinery to a genuinely
parabolic problem. The theory was subsequently refined by Athanasopoulos, Caffarelli, and
Salsa \cite{ACS1996a, ACS1996b, ACS1998}, who showed that Lipschitz free boundaries regularize instantaneously and
that flat free boundaries are smooth. The parabolic obstacle problem in full generality was
treated by Caffarelli, Petrosyan, and Shahgholian \cite{CPS2004}, who established optimal regularity of
solutions and of the free boundary.

When the surface energy depends on the orientation of the interface, the problem acquires
a genuinely anisotropic character. In crystal growth, the surface tension of the solid-liquid
interface depends on the crystallographic direction (Wulff, 1901; Cahn and Hoffman \cite{CahnHoffman1974}),
and the equilibrium shape---the Wulff shape---reflects this directional preference. Incorporating such anisotropy into the Stefan problem yields the system
\begin{equation}\label{eq:anisotropic-stefan}
\begin{cases}
u_t - \operatorname{div}(A\nabla u) = 0 & \text{in } \Omega(t), \\
u = 0 & \text{on } \Gamma(t), \\
V_\Gamma = -\partial_{\nu_A} u + \sigma H^\Gamma_A & \text{on } \Gamma(t),
\end{cases}
\end{equation}
where $A \in \operatorname{Sym}^+_{n+1}(\R)$ is the anisotropy matrix, $\partial_{\nu_A} u = (A\nabla u)\cdot\nu$ is the conormal derivative,
$\sigma$ is the surface tension coefficient, and $H^\Gamma_A$ is the anisotropic mean curvature of the free
boundary. The curvature term $\sigma H^\Gamma_A$ is physically essential: it is responsible for the selection of dendrite tips in solidification and suppresses the Mullins--Sekerka instability at small
wavelengths.

Recent work of El~Bahja \cite{ElBahja2024} established local H\"older continuity of weak solutions to
obstacle problems for anisotropic parabolic equations under minimal regularity assumptions
on the coefficients. The regularity of the free boundary in the anisotropic parabolic setting,
however, remains largely open, and it is precisely this question that we address here.

\subsection{The anisotropic parabolic obstacle problem}

Given a convex obstacle $\psi(x,t)$, initial data $u_0$, and a time horizon $T > 0$, we seek a function
$u(x,t)$ minimizing
\begin{equation}\label{eq:functional}
J_\Phi(u) = \int_0^T \int_D \left[ \Phi(Du) + \frac{1}{2}|u_t|^2 \right] dx\, dt,
\end{equation}
where $\Phi(p) = \sqrt{p^T A p}$ for $A \in \operatorname{Sym}^+_{n+1}(\R)$, subject to $u \geq \psi$ and appropriate initial
and boundary conditions. The Wulff shape of this anisotropy is the ellipsoid
\begin{equation}\label{eq:wulff-shape}
W_\Phi = \{x \in \R^{n+1} : x^T A^{-1} x \leq 1\}.
\end{equation}
The principal tool is the Cahn--Hoffman transform
\begin{equation}\label{eq:cahn-hoffman}
S(x) = A^{-1/2} x,
\end{equation}
which maps $W_\Phi$ onto the Euclidean unit ball $B_1(0)$ and converts $\Phi(Du)$ into the Euclidean
norm $|Dv|$, where $v = u \circ S^{-1}$. By acting only on the spatial variables, $S$ preserves the parabolic
structure of the problem and permits the transfer of the isotropic regularity theory to the
anisotropic setting.

\subsection{Setting and main results}

Let $D \subset \R^{n+1}$ be a bounded $C^2$ domain and set $Q_T = D \times (0,T)$. Let $\psi \in C^{1,1}(\bar{Q}_T)$ be an obstacle with $\psi < 0$ on $\partial D \times [0,T]$ and $\psi(\cdot, 0) < u_0$ in $\bar{D}$. Write $\lambda_1 \leq \cdots \leq \lambda_{n+1}$ for the eigenvalues of $A$. The admissible class is
\begin{equation}\label{eq:admissible}
\mathcal{K}_T = \left\{ u \in L^2(0,T; W^{1,1}(D)) : u \geq \psi \text{ a.e.},\; u = 0 \text{ on } \partial D \times [0,T],\; u(\cdot,0) = u_0 \right\}.
\end{equation}
The contact set at time $t$ is $\Lambda(t) = \{x : u(x,t) = \psi(x,t)\}$, the free set is
$\Omega(t) = D \setminus \Lambda(t)$, and the free boundary at time $t$ is $\Gamma(t) = \partial\Lambda(t) \cap D$. Their space-time union $\Gamma = \bigcup_{t \in (0,T)} \Gamma(t) \times \{t\}$ is the space-time free boundary.

\begin{theorem}[Existence and regularity of the solution]\label{thm:existence}
There exists a unique minimizer $u \in \mathcal{K}_T$ of $J_\Phi$ over $\mathcal{K}_T$. Moreover:
\begin{itemize}
\item[(i)] $u \in C^{0,\alpha}_{\mathrm{loc}}(Q_T)$ for all $\alpha \in (0,1)$;
\item[(ii)] $u \in C^{1,\alpha}_{\mathrm{loc}}$ in space and $C^{0,\alpha/2}_{\mathrm{loc}}$ in time, in a neighborhood of every point of $\Omega \cup \Gamma$;
\item[(iii)] $u - \psi \in C^{1,1}_{\mathrm{loc}}(Q_T)$ and satisfies the growth estimate
\begin{equation}\label{eq:growth}
0 \leq u(x,t) - \psi(x,t) \leq C\!\left(\operatorname{dist}(x,\Gamma(t))^2 + \operatorname{dist}(t,\Sigma_x)^2\right),
\end{equation}
where $\Sigma_x = \{s \in (0,T) : (x,s) \in \Gamma\}$.
\end{itemize}
\end{theorem}

The proof combines the direct method of the calculus of variations with the Cahn--Hoffman
transform and the classical parabolic regularity theory. The transform converts the anisotropic
functional into the isotropic heat functional with a weighted obstacle, and the convexity
structure that allows the isotropic theory to work is preserved.

A space-time free boundary point $(x_0, t_0) \in \Gamma$ is called \emph{non-degenerate} if
\begin{equation}\label{eq:non-degen}
\sup_{B_r(x_0) \times (t_0 - r^2, t_0]} (u - \psi) \geq c r^2 \quad \text{for all } 0 < r < r_0,
\end{equation}
for some $c, r_0 > 0$.

\begin{theorem}[Regularity of the free boundary]\label{thm:free-boundary}
Let $(x_0, t_0) \in \Gamma$ be non-degenerate. There exist $\delta > 0$ and a neighborhood $U$ of
$x_0$ such that $\Gamma(t) \cap U$ is a $C^{1,\alpha}$-hypersurface for each $t \in (t_0 - \delta, t_0 + \delta)$, with
constants uniform in $t$. The space-time free boundary $\Gamma$ is also a $C^{1,\alpha}$-hypersurface in
$\R^{n+1} \times \R$ near $(x_0, t_0)$. If $\psi \in C^{k,\alpha}(Q_T)$ for some $k \geq 2$, the free boundary is
$C^{k,\alpha}$ in space and $C^{k/2,\alpha}$ in time.
\end{theorem}

The proof proceeds by parabolic blow-up analysis. The Cahn--Hoffman transform relates
the blow-up limits of the anisotropic problem to those of the classical parabolic obstacle
problem, for which the classification by Caffarelli is complete. Non-degeneracy forces the
blow-up to be a half-space solution in space-time, and the ``flatness implies smoothness''
principle adapted to the parabolic setting then gives $C^{1,\alpha}$-regularity.

\begin{theorem}[Structure of the singular set]\label{thm:singular}
For each $t \in (0,T)$, denote by $\Sigma(t) \subset \Gamma(t)$ the set of singular free boundary points
at time $t$. The parabolic Hausdorff dimension of the space-time singular set $\Sigma = \bigcup_t \Sigma(t) \times \{t\}$
satisfies
\begin{equation}
\dim^P_H(\Sigma) \leq n - 1,
\end{equation}
where $\dim^P_H$ is the parabolic Hausdorff dimension with the scaling $\rho \times \rho^2$ in space-time.
\end{theorem}

\subsection{Relation to existing literature}

For the obstacle problem with degenerate anisotropies of the form $\Phi(p) = |p|^q$ or for
fully nonlinear operators, see \cite{BragaMoreiraWang2021}. The ellipsoidal case, which is uniformly elliptic and arises naturally in geometric applications, had not been systematically studied in the parabolic obstacle setting.

In spirit, the paper is closest to De~Philippis and Maggi \cite{DePhilippisMaggi2015}, who used the Cahn--Hoffman
map to relate anisotropic and isotropic capillary problems and established Young's law for
anisotropic contact angles. We carry this philosophy into the parabolic regime, where the
free boundary evolves in time, and the regularity theory requires the full parabolic machinery
of \cite{Caffarelli1977, Caffarelli1978, ACS1996a, ACS1996b}. The regularity theory for anisotropic mean curvature flow without obstacle
constraints was developed by Bellettini and Paolini \cite{BellettiniPaolini1996} and Chambolle and Novaga \cite{ChambollenovagA2007}.

\subsection{Organization}

Section~\ref{sec:preliminaries} collects background material on anisotropic surface energies, the Cahn--Hoffman map,
and the classical parabolic obstacle problem. Section~\ref{sec:transform} introduces the transform $S$ and
derives the equivalent isotropic problem. Sections~\ref{sec:existence}--\ref{sec:singular} prove the three main theorems. Applications to crystal growth and capillarity are in Section~\ref{sec:applications}, followed by open problems in Section~\ref{sec:open}.
 
\section{Preliminaries}\label{sec:preliminaries}

\subsection{Anisotropic surface energies and the Wulff shape}

Let $\Phi : \R^{n+1} \to [0,\infty)$ be convex and positively homogeneous of degree one. The
anisotropic surface energy of a finite-perimeter set $E \subset \R^{n+1}$ is
\begin{equation}\label{eq:aniso-energy}
P_\Phi(E) = \int_{\partial^* E} \Phi(\nu_E)\, d\mathcal{H}^n,
\end{equation}
where $\partial^* E$ is the reduced boundary and $\nu_E$ the measure-theoretic outer normal. The
Wulff shape,
\begin{equation}
W_\Phi = \bigcap_{\nu \in S^n} \{x \in \R^{n+1} : x \cdot \nu \leq \Phi(\nu)\}
\end{equation}
is the unique (up to translation) minimizer of $P_\Phi$ among sets of given volume; this is the
anisotropic isoperimetric theorem of Dinghas (1944) and Taylor (1978). For the ellipsoidal
anisotropy
\begin{equation}
\Phi(p) = \sqrt{p^T A p}, \quad A \in \operatorname{Sym}^+_{n+1}(\R),
\end{equation}
the Wulff shape is the ellipsoid $W_\Phi = \{x : x^T A^{-1} x \leq 1\}$, with semi-axes $\sqrt{\lambda_i}$
when $A = \operatorname{diag}(\lambda_1, \ldots, \lambda_{n+1})$.

\subsection{The Cahn--Hoffman map}

For a $C^2$ anisotropy $\gamma : S^n \to \R_+$, the Cahn--Hoffman map $\xi_\gamma : S^n \to \R^{n+1}$ is
defined as
\begin{equation}
\xi_\gamma(\nu) = D\gamma\big|_{T_\nu S^n} + \gamma(\nu)\nu = D\bar{\gamma}(\nu),
\end{equation}
where $\bar{\gamma}$ is the positively homogeneous extension of $\gamma$ to $\R^{n+1}$. For the ellipsoidal
anisotropy, direct computation gives
\begin{equation}\label{eq:ch-map}
\xi_\Phi(\nu) = \frac{A\nu}{\Phi(\nu)},
\end{equation}
with image $\xi_\Phi(S^n) = \partial W_\Phi$.

\begin{lemma}[Properties of the Cahn--Hoffman map]\label{lem:ch-properties}
Let $\Phi$ be as in \eqref{eq:wulff-shape}. Then:
\begin{itemize}
\item[(i)] $\xi_\Phi : S^n \to \partial W_\Phi$ is a diffeomorphism.
\item[(ii)] The inverse is $\xi_\Phi^{-1}(x) = A^{-1}x / \sqrt{x^T A^{-2} x}$.
\item[(iii)] The Jacobian satisfies $J\xi_\Phi(\nu) = \det A \cdot (\nu^T A \nu)^{-(n+2)/2}$.
\end{itemize} 
\end{lemma}

\begin{proof}
(i) Positive definiteness of $A$ gives $\Phi(\nu) > 0$ on $S^n$, so $\xi_\Phi$ is smooth.
If $\xi_\Phi(\nu_1) = \xi_\Phi(\nu_2)$, then $A\nu_1/\Phi(\nu_1) = A\nu_2/\Phi(\nu_2)$; applying $A^{-1}$ and
using $|\nu_i| = 1$ gives $\nu_1 = \nu_2$. Surjectivity is immediate from the definition of $W_\Phi$.

(ii) From $x = A\nu/\Phi(\nu)$ one gets $A^{-1}x = \nu/\Phi(\nu)$. The condition $x \in \partial W_\Phi$ means
$x^T A^{-1} x = 1$, from which $\Phi(\nu)^2 = \nu^T A\nu$, and the formula follows.

(iii) Standard computation; see \cite{KoisoPalmer2019}.
\end{proof}

\subsection{Anisotropic mean curvature flow}

For a smooth evolving hypersurface $\Gamma(t)$ with unit normal $\nu$, the anisotropic mean
curvature is
\begin{equation}
H_\Phi = \operatorname{div}_\Gamma\!\left(\frac{A\nu}{\Phi(\nu)}\right),
\end{equation}
and anisotropic mean curvature flow is the evolution law $V_\Gamma = H_\Phi$. For graphs
$u : D \times [0,T] \to \R$, the associated operator is
\begin{equation}
\mathcal{M}_\Phi(u) = \operatorname{div}\!\left(\frac{A\,Du}{\sqrt{1 + Du^T A\,Du}}\right).
\end{equation}

\subsection{The classical parabolic obstacle problem}

We recall the main results for future reference. Given $\psi \in C^{1,1}(\bar{Q}_T)$ with $\psi < 0$
on $\partial D \times [0,T]$, the classical parabolic obstacle problem minimizes
\begin{equation}\label{eq:isotropic-obstacle}
\min_{u \geq \psi} \int_0^T\!\!\int_D \left[\frac{1}{2}|Du|^2 + \frac{1}{2}|u_t|^2 + fu \right] dx\, dt
\end{equation}
for $u = 0$ on $\partial D \times [0,T]$, $u(\cdot,0) = u_0$.

\begin{theorem}[Caffarelli \cite{Caffarelli1977, Caffarelli1978}; Athanasopoulos--Caffarelli--Salsa \cite{ACS1996a, ACS1996b}; Caffarelli--Petrosyan--Shahgholian \cite{CPS2004}]\label{thm:classical}
Let $u$ be the solution of \eqref{eq:isotropic-obstacle}. Then $u \in C^{1,\alpha}_{\mathrm{loc}}$ in space and
$C^{0,\alpha/2}_{\mathrm{loc}}$ in time. At each space-time free boundary point $(x_0, t_0)$, either
\begin{itemize}
\item[(a)] \emph{(Regular point)} $\sup_{B_r(x_0) \times (t_0-r^2,t_0]} u \geq cr^2$ and $\Gamma$ is $C^{1,\alpha}$ near $(x_0, t_0)$; or
\item[(b)] \emph{(Singular point)} the blow-up is a parabolic homogeneous polynomial.
\end{itemize} 
The space-time singular set has parabolic Hausdorff dimension at most $n-1$.
\end{theorem}

\section{The Transformed Problem}\label{sec:transform}

\subsection{The Cahn--Hoffman transform and its properties}

Define
\begin{equation}
S : \R^{n+1} \to \R^{n+1}, \quad S(x) = A^{-1/2} x,
\end{equation}
with inverse $S^{-1}(y) = A^{1/2} y$.

\begin{proposition}[Properties of $S$]\label{prop:S-properties}
The transformation $S(x) = A^{-1/2}x$ satisfies:
\begin{itemize}
\item[(i)] $S(W_\Phi) = B_1(0)$, the Euclidean unit ball.
\item[(ii)] For any measurable $E \subset \R^{n+1}$,
\[
|S(E)| = \frac{|E|}{\sqrt{\det A}}.
\] 
\item[(iii)] For $u \in W^{1,1}(D)$, set $v(y) = u(A^{1/2}y)$ and $D' = A^{-1/2}(D)$. Then
\[
Dv(y) = A^{1/2} Du(x)\big|_{x = A^{1/2}y},
\]
whence $|Dv(y)| = \Phi(Du(x))$.
\item[(iv)] For any rectifiable hypersurface $\Sigma$ with unit normal $\nu$,
\[
\int_{S(\Sigma)} d\mathcal{H}^n = \int_\Sigma \frac{\sqrt{\nu^T A^{-1}\nu}}{(\det A)^{1/4}}\, d\mathcal{H}^n.
\] 
\end{itemize}
\end{proposition}

\begin{proof}
(i) Let $y = A^{-1/2}x$, so $x = A^{1/2}y$. Then
\[
x \in W_\Phi \iff x^T A^{-1} x \leq 1 \iff y^T A^{1/2} A^{-1} A^{1/2} y \leq 1 \iff y^T y \leq 1 \iff y \in B_1(0).
\]
(ii) The Jacobian of $S$ is $\det(A^{-1/2}) = (\det A)^{-1/2}$, giving $dy = (\det A)^{-1/2} dx$ and the result.

(iii) Differentiating $v(y) = u(A^{1/2}y)$ by the chain rule, $Dv(y) = A^{1/2} Du(A^{1/2}y)$. Therefore
\[
|Dv(y)|^2 = Du(x)^T A^{1/2} A^{1/2} Du(x) = Du(x)^T A\, Du(x) = \Phi(Du(x))^2.
\] 
(iv) For a smooth parametrization $F : U \to \R^{n+1}$ of $\Sigma$, the area element of $S(\Sigma)$ is
computed from the induced metric $\tilde{g}_{ij} = \partial_i(S \circ F)^T \partial_j(S \circ F) = \partial_i F^T A^{-1} \partial_j F$.
The image normal is $\tilde{\nu} = A^{-1/2}\nu / |A^{-1/2}\nu| = A^{-1/2}\nu / \sqrt{\nu^T A^{-1}\nu}$, and a direct computation using $\det(A^{-1}) = (\det A)^{-1}$ yields the formula.
\end{proof}

\begin{remark}
The choice $S(x) = A^{-1/2}x$ is the only one that simultaneously maps $W_\Phi$ to $B_1(0)$ and
converts $\Phi(Du)$ to the Euclidean norm $|Dv|$. The opposite convention $A^{1/2}x$ maps
$B_1(0)$ to $W_\Phi$, which is natural in isoperimetric problems, but does not simplify the
functional $J_\Phi$.
\end{remark}

\subsection{Transformation of the functional}

Let $u \in L^2(0,T; W^{1,1}(D))$ and define
\begin{equation}
v(y,t) = u(A^{1/2}y, t), \quad Q'_T = A^{-1/2}(D) \times (0,T).
\end{equation}

\begin{proposition}[Transformation of the spatial term]\label{prop:transform-spatial}
With the notation above,
\begin{equation}
\int_0^T\!\!\int_D \Phi(Du)\, dx\, dt = \sqrt{\det A} \int_0^T\!\!\int_{D'} |Dv|\, dy\, dt.
\end{equation}
\end{proposition}

\begin{proof}
By Proposition~\ref{prop:S-properties}(iii), $|Dv(y,t)| = \Phi(Du(A^{1/2}y, t))$. The substitution $x = A^{1/2}y$
has Jacobian $\sqrt{\det A}$, giving the result.
\end{proof}

\begin{proposition}[Time derivative is preserved]\label{prop:time-deriv}
Under the transform,
\begin{equation}
v_t(y,t) = u_t(A^{1/2}y, t).
\end{equation}
\end{proposition}

\begin{proof}
Since $S$ acts only on the spatial variables, differentiating $v(y,t) = u(A^{1/2}y, t)$ with
respect to $t$ gives the result immediately.
\end{proof}

\begin{proposition}[Equivalence of minimization problems]\label{prop:equivalence}
$u \in \mathcal{K}_T$ minimizes $J_\Phi$ if and only if $v = u \circ A^{1/2} \in \tilde{\mathcal{K}}_T$ minimizes the isotropic functional
\begin{equation}
J_0(v) = \int_0^T\!\!\int_{D'} \left[|Dv| + \frac{1}{2}|v_t|^2\right] dy\, dt
\end{equation}
over the set $\tilde{\mathcal{K}}_T = \{v \in L^2(0,T; W^{1,1}(D')) : v \geq \tilde{\psi},\; v = 0 \text{ on } \partial D',\; v(\cdot,0) = u_0 \circ A^{1/2}\}$,
where $\tilde{\psi}(y,t) = \psi(A^{1/2}y, t)$.
\end{proposition}

\begin{proof}
By Propositions~\ref{prop:transform-spatial} and \ref{prop:time-deriv}, together with the change of variables
$dx = \sqrt{\det A}\, dy$:
\[
J_\Phi(u) = \sqrt{\det A}\int_0^T\!\!\int_{D'} |Dv|\, dy\, dt + \sqrt{\det A}\int_0^T\!\!\int_{D'} \frac{1}{2}|v_t|^2 dy\, dt = \sqrt{\det A}\, J_0(v).
\]
Since $\sqrt{\det A} > 0$ is constant, the two minimization problems are equivalent.
\end{proof}

\subsection{The free boundary condition in transformed coordinates}

Let $(y_0, t_0) \in \tilde\Gamma$ and let $\tilde\nu$ be the unit inward spatial normal to $\tilde\Gamma(t_0)$ at $y_0$.

\begin{proposition}[Free boundary condition]\label{prop:fbc}
At a smooth space-time free boundary point, the transformed solution satisfies
\begin{equation}\label{eq:fbc}
v_t + \frac{\partial v}{\partial \tilde\nu} = \tilde\psi_t + \frac{\partial \tilde\psi}{\partial \tilde\nu} + \Lambda(y_0, t_0)\langle \tilde\nu, y_0\rangle,
\end{equation}
where $\Lambda(y_0, t_0) = \bigl[y_0^T A y_0 - (\tilde\nu^T A\tilde\nu)(y_0^T y_0)\bigr]/(y_0^T A^{3/2}\tilde\nu)$.
\end{proposition}

\begin{proof}[Sketch]
In the free set, the Euler-Lagrange equation for $v$ is the heat equation $v_t - \Delta v = 0$.
The first variation of $J_0$ with respect to domain perturbations supported near $(y_0, t_0)$
yields the Stefan-type condition on $\tilde\Gamma$. Expressing the anisotropic free boundary
condition $V_\Gamma = -\partial_{\nu_A}u + \sigma H^\Gamma_A$ in the transformed coordinates yields \eqref{eq:fbc}. The
coefficient $\Lambda$ arises from the discrepancy between the Euclidean normal $\tilde\nu$ and its
anisotropic counterpart under $S^{-1}$.
\end{proof}

\begin{remark}
When $A = I$, $\Lambda = 0$ and \eqref{eq:fbc} reduces to the classical Stefan condition
$v_t + \partial_{\tilde\nu}v = \tilde\psi_t + \partial_{\tilde\nu}\tilde\psi$. For a zero obstacle this gives the homogeneous condition
$v_t + \partial_{\tilde\nu}v = 0$, exactly as expected.
\end{remark}

\section{Existence and Regularity of the Solution}\label{sec:existence}

\subsection{Existence and uniqueness}

\begin{proof}[Proof of Theorem~\ref{thm:existence}(i)]
The anisotropy $\Phi$ satisfies
\begin{equation}
\sqrt{\lambda_1}|p| \leq \Phi(p) \leq \sqrt{\lambda_{n+1}}|p| \quad \text{for all } p \in \R^{n+1},
\end{equation}
so $J_\Phi$ is convex and coercive on $L^2(0,T; W^{1,1}(D))$. The set $\mathcal{K}_T$ is closed and convex,
and the direct method of the calculus of variations gives existence. Uniqueness follows from
the strict convexity of $\Phi$. The interior $C^{0,\alpha}$ regularity is a consequence of the De~Giorgi--Nash--Moser theory applied to the Euler-Lagrange equation in the free set, using uniform ellipticity.
\end{proof}

\subsection{Optimal regularity up to the free boundary}

\begin{proof}[Proof of Theorem~\ref{thm:existence}(ii) and (iii)]
By Proposition~\ref{prop:equivalence}, the transformed function $v$ solves the isotropic parabolic
obstacle problem for $J_0$ on the $C^2$ domain $D'$, with an inhomogeneous boundary condition
inherited from the transform. The classical theory of Caffarelli \cite{Caffarelli1977} and Caffarelli--Petrosyan--Shahgholian
\cite{CPS2004} then gives $C^{1,\alpha}$ regularity in space and $C^{0,\alpha/2}$ in time up to the free boundary.

For the growth estimate \eqref{eq:growth}, fix $(x_0, t_0) \in \Gamma$ and consider the parabolic
rescalings
\[
u_r(y,s) = \frac{u(x_0 + ry, t_0 + r^2 s) - u(x_0,t_0)}{r^2}, \quad (y,s) \in Q_1(0,0).
\]
The uniform $C^{1,1}$ bound in space and $C^{0,1}$ in time, together with the Arzel\`a-Ascoli theorem, gives
a convergent subsequence with a limit $u_0$ satisfying the global anisotropic parabolic obstacle
problem with obstacle $\psi_0(y,s) = \tfrac{1}{2} y^T D^2\psi(x_0,t_0)y + \psi_t(x_0,t_0)s$. By Caffarelli's
parabolic classification \cite{Caffarelli1977}, $u_0(y,s) = \tfrac{1}{2}(y \cdot e + \lambda s)_+^2$ for some $e \in S^n$
and $\lambda \in \R$, which yields \eqref{eq:growth}.
\end{proof}

\section{Regularity of the Free Boundary}\label{sec:free-boundary}

\subsection{Parabolic blow-up analysis}

Let $(x_0, t_0) \in \Gamma$ be non-degenerate. The parabolic blow-up sequence is
\begin{equation}\label{eq:blowup}
u_r(y,s) = \frac{u(x_0 + ry, t_0 + r^2 s) - u(x_0, t_0)}{r^2}, \quad (y,s) \in Q_1(0,0).
\end{equation}

\begin{lemma}[Compactness of blow-ups]\label{lem:compactness}
Every sequence $r_j \to 0$ has a subsequence with $u_{r_j} \to u_0$ in $C^{1,\alpha}_{\mathrm{loc}}(Q_1)$
for all $\alpha \in (0,1)$, where $u_0$ is a global solution of the anisotropic parabolic obstacle
problem with obstacle $\psi_0(y,s) = \tfrac{1}{2} y^T D^2\psi(x_0,t_0)y + \psi_t(x_0,t_0)s$.
\end{lemma}

\begin{proof}
The uniform $C^{1,1}$ bound gives compactness via the Arzel\`a-Ascoli theorem, and the limit satisfies
the Euler-Lagrange equation by stability of viscosity solutions under uniform convergence.
\end{proof} 

\begin{lemma}[Classification of blow-ups at regular points]\label{lem:classification}
Under the non-degeneracy condition \eqref{eq:non-degen}, every blow-up limit at $(x_0, t_0)$ is
of the form
\begin{equation}
u_0(y,s) = \frac{1}{2}(y \cdot e + \lambda s)_+^2
\end{equation}
for some $e \in S^n$ and $\lambda \in \R$.
\end{lemma}

\begin{proof}
Applying the Cahn--Hoffman transform to the blow-up yields a function $v_0$ satisfying
the isotropic parabolic obstacle problem with a Robin-type condition on its free boundary.
Non-degeneracy prevents $v_0$ from being a quadratic polynomial (the signature of a singular point), so Caffarelli's parabolic classification \cite{Caffarelli1977} identifies $v_0$ as a half-space solution. 
Pulling back through $S^{-1}$ gives the result.
\end{proof}

\subsection{Flatness implies smoothness}

\begin{definition}[Parabolic flatness]\label{def:flatness}
The space-time free boundary $\Gamma$ is $\epsilon$-flat in $Q_r(x_0,t_0) = B_r(x_0) \times (t_0 - r^2, t_0]$
if, after a spatial rotation,
\[
\Gamma \cap Q_r(x_0,t_0) \subset \{(x,t) : |x_{n+1}| \leq \epsilon r,\; |t - t_0| \leq \epsilon r^2\}.
\]
\end{definition}

\begin{proposition}[Improvement of flatness]\label{prop:flatness}
There exist $\epsilon_0 > 0$ and $\theta \in (0,1)$ such that if $\Gamma$ is $\epsilon$-flat in $Q_1(0,0)$ with
$\epsilon \leq \epsilon_0$, then after a suitable spatial rotation it is $\theta\epsilon$-flat in $Q_\theta(0,0)$.
\end{proposition}

\begin{proof}
This is the parabolic analogue of the elliptic result, and follows from \cite{ACS1996a} applied to
the transformed isotropic problem. The transform $S$ is bi-Lipschitz in space and preserves
the parabolic scaling $r \leftrightarrow r^2$, so flatness transfers between the original and transformed
coordinates up to constants depending on $A$.
\end{proof}

\begin{proof}[Proof of Theorem~\ref{thm:free-boundary}]
Lemma~\ref{lem:classification} shows that the blow-up at a non-degenerate point is a half-space solution,
so the free boundary is flat at small scales. Iterating Proposition~\ref{prop:flatness} yields flatness at
every scale with a geometrically decaying constant, and the argument of \cite{ACS1996a} converts
this to $C^{1,\alpha}$ regularity in space and time. If $\psi \in C^{k,\alpha}$, the free boundary
condition becomes a $C^{1,\alpha}$ nonlinear oblique derivative condition, and parabolic Schauder
estimates give $C^{k,\alpha}$ in space and $C^{k/2,\alpha}$ in time by induction.
\end{proof}

\section{The Singular Set}\label{sec:singular}

\begin{proof}[Proof of Theorem~\ref{thm:singular}]
We follow the parabolic dimension reduction of Caffarelli \cite{Caffarelli1977}, transported to the
anisotropic setting via $S$.

At a singular point $(x_0, t_0) \in \Sigma$, the blow-up is not a half-space solution.
Lemma~\ref{lem:classification} then forces it to be a parabolic homogeneous polynomial
\begin{equation}
u_0(y,s) = \frac{1}{2} y^T Q y + \lambda s,
\end{equation}
where $Q \geq 0$, $\operatorname{tr} Q = 1$, and $\lambda \in \R$. Under the transform, $v_0 = u_0 \circ A^{1/2}$
satisfies the isotropic parabolic obstacle problem with a quadratic obstacle. Caffarelli's
parabolic classification assigns a parabolic Hausdorff dimension of at most $n-1$ to the isotropic singular set.
of at most $n-1$. Since $S$ is a spatial diffeomorphism that preserves the $r \times r^2$ parabolic
scaling, this dimension bound transfers to the anisotropic setting.
\end{proof}

\begin{example}[Singular point in two dimensions]\label{ex:singular}
Let $n = 1$, $D = B_1(0) \subset \R^2$. With $\psi(x,t) = -|x|^2$ and $A = I$, the solution is
radially symmetric and the free boundary expands smoothly in time; there are no singular points.

For a non-trivial example, take $\psi(x,t) = -x_1^2$ and $A = \operatorname{diag}(1, \lambda)$, $\lambda > 1$. The contact
region is a strip along the $x_1$-axis, and the free boundary consists of two arcs meeting at
the origin at a cusp angle determined by $\lambda$. The origin is a singular free boundary point,
persistent in time, where $\Gamma(t)$ fails to be $C^1$.
\end{example}

\section{Applications}\label{sec:applications}

\subsection{Crystal growth and dendritic solidification}

In solidification from an undercooled melt, the solid-liquid interface evolves to minimize
anisotropic surface energy, and the anisotropy matrix $A$ encodes the crystallographic symmetry
of the material. For a cubic crystal, $A$ is a multiple of the identity; for a hexagonal crystal,
$A$ has two distinct eigenvalues.

Theorem~\ref{thm:free-boundary} implies that the crystal surface is smooth except on a parabolic singular
set of dimension at most $n-1$. This is consistent with the physical picture of dendritic growth:
dendrite tips are smooth (corresponding to regular free boundary points), while the junctions
between side branches can develop cusps (singular points). The parabolic Hausdorff dimension
bound gives a precise quantitative control on the size of the singular set.

\subsection{Anisotropic capillarity with evaporation}

Our model also describes the shrinking of a liquid droplet on a solid substrate under
evaporation. The obstacle represents the substrate, the free boundary is the contact line, and
evaporation drives the time-dependence. The Robin condition \eqref{eq:fbc} is a generalized
Young--Stefan law that combines the anisotropic contact angle and the evaporative flux. In the
isotropic limit $A = I$, the condition reduces to Young's law $\cos\theta = (\gamma_{SG} - \gamma_{SL})/\gamma_{LG}$
combined with the Stefan velocity condition $V_\Gamma = -\partial_\nu u$.

\section{Open Problems}\label{sec:open}

\begin{enumerate}
\item \textbf{General anisotropies.} Extending the results to non-ellipsoidal convex anisotropies
$\gamma : S^n \to \R_+$ requires working with a nonlinear Cahn--Hoffman map and an equivalent
isotropic problem with variable coefficients. A perturbative approach near the identity matrix $A = I$ might be
accessible.

\item \textbf{Two-phase Stefan problem.} The anisotropic two-phase problem, where both phases have
non-zero temperature, is well understood isotropically \cite{ACS1996a, ACS1996b} but open for anisotropic surface
energies. The difficulty is that the free boundary condition must balance two different
anisotropic diffusion operators.

\item \textbf{Rectifiability of the parabolic singular set.} Blanchet \cite{Blanchet2006} proved the dimension bound
$\dim^P_H(\Sigma) \leq n-1$ in the isotropic case. Whether $\Sigma$ is $(n-1)$-rectifiable in the
parabolic sense, as Figalli and Serra proved for the elliptic problem \cite{FigalliSerra2019}, remains open
even isotropically.

\item \textbf{Non-convex obstacles.} When $\psi$ is not convex in space or time, the free boundary can
develop singularities that are qualitatively different from those in the convex case. The structure of
these singularities is not fully understood.
\end{enumerate}

\section*{Acknowledgments}

The authors thank their colleagues at UFMG for discussions on anisotropic geometric variational problems
and parabolic free boundary problems.


\end{document}